\documentclass[12pt]{amsart}
\newcounter{defcounter}
\usepackage{amsmath,amssymb}

\usepackage[all]{xy}
\usepackage{pdfsync}
\usepackage[mathscr]{euscript}

\theoremstyle{plain}
\newtheorem{theorem}{Theorem}
\newtheorem{proposition}[theorem]{Proposition}

\newtheorem{lemma}[theorem]{Lemma}

\newtheorem{theoremalpha}{Theorem}

\newtheorem{conjecturealpha}[theoremalpha]{Conjecture}

\theoremstyle{definition}
\newtheorem{definition}[theorem]{Definition}

\newtheorem*{examplenn}{Example}

\newtheorem{remark}[theorem]{Remark}

\newcommand{\lra}{\longrightarrow}

\newcommand{\OO}{\mathcal{O}}

\numberwithin{theorem}{section}

\usepackage{tikz-cd}
\usepackage{verbatim}

\tikzset{
  symbol/.style={
    draw=none,
    every to/.append style={
      edge node={node [sloped, allow upside down, auto=false]{$#1$}}}
  }
}

\begin{document}

\title{Projective normality of canonical symmetric squares}

\author{John Sheridan}
  \address{Department of Mathematics, Princeton University, Princeton, New Jersey  08544}
\email{{\tt john.sheridan@princeton.edu}}

\maketitle

\section*{Introduction}

Let $C$ be a smooth complex projective curve of genus $g \geq 2$, and consider the canonical mapping $C \lra \mathbb{P}^{g-1}$ of $C$. Recall that this is an embedding when $C$ is non-hyperelliptic, in which case Noether's theorem asserts that $C \subseteq \mathbb{P}^{g-1}$ is projectively normal. Petri proved that the homogeneous ideal $I_{C/\mathbb{P}^{g-1}}$ is generated by quadrics unless $C$ is trigonal or a smooth plane quintic (i.e. unless $\text{Cliff}(C) = 1$ --- see \S\ref{setup}), and a famous conjecture of Green predicts how these statements should generalize to higher syzygies.

It is natural to ask whether any of these statements extend to surfaces of higher dimensional varieties. Of course one could look at the canonical maps of arbitrary surfaces, but in this generality one doesn't expect precise statements. On the other hand, consider the symmetric square $C_2$ of $C$ which maps canonically to $\mathbb{P}^{{g\choose 2}-1}$ (see \S \ref{setup-sym}). One might hope that the algebraic properties of this embedding combine, in an interesting way, the geometry of curves with two-dimensional syzygetic questions.

In this direction, our main result is an analogue of Noether's theorem.

\begin{theoremalpha}\label{main-theorem}
The canonical map of $C_2$ is a projectively normal embedding if and only if $C$ is neither hyperelliptic, trigonal nor a smooth plane quintic.
\end{theoremalpha}

Note that here, via the result of Petri mentioned above, properties of the defining equations of $C$ are reflecting the projective normality of $C_2$. We conjecture that a similar picture continues for higher syzygies and higher dimensions. Namely, say that a line bundle $L$ on $C$ satisfies \textit{syzygy shifting (to dimension $k$)} if, whenever $L$ has property $N_p$ (i.e. has linear syzygies to order $p \geq 0$ --- see \S\ref{setup}), the associated line bundle $N_L := \text{det}(L^{[k]})$ on $C_k$ has property $N_{p-(k-1)}$.

\begin{conjecturealpha}[Syzygy shifting]\label{syzygy-shifting}
The canonical bundle $K_C$ on $C$ satisfies syzygy shifting to dimension $k$ when $\textnormal{Cliff}(C) \geq k$.
\end{conjecturealpha}

Our next result refines the backward direction of Theorem \ref{main-theorem} in the spirit of \cite[Conjecture 3.4]{green-laz} and is suggestive of what one might hope for in terms of syzygy shifting for non-canonical line bundles.

\begin{theoremalpha}\label{more-general}
If $L$ is a very ample line bundle on $C$ determining an embedding for which the homogeneous ideal $I_C$ is generated by quadrics, and \[\textnormal{deg}(L) \geq 2g + 2 - 2\hspace{1pt}h^1(C,L) - \textnormal{Cliff}(C),\] then the line bundle $N_L$ on $C_2$ determines a projectively normal embedding.
\end{theoremalpha}

The thinking behind the optimistic Conjecture \ref{syzygy-shifting} is that the property of being $p$-very-ample seems to shift by the dimension jump (see the Example below) so perhaps the related property $N_p$ will too. The following example, which readily generalizes, illustrates\footnote{to be completely precise about this, one needs to deal with non-reduced point sets in $C$ as well.} how $3$-very-ampleness of $\mathcal{O}_C(1)$ ``shifts by 1" to $2$-very-ampleness of $\mathcal{O}_{C_2}(1)$. It applies to the present setting since the canonical map of $C_2$ is in fact a secant map as described below (again, see \S\ref{setup}).

\begin{examplenn}\label{example1}
For an embedding $C \subset \mathbb{P} = \mathbb{P}V$, consider the \textit{secant map}
\begin{center}
    \begin{tikzcd}[row sep=tiny]
    C_2 \ar[r,dashrightarrow] & \mathbb{G}(1,\mathbb{P}) \subseteq \mathbb{P}\wedge^2V\\[-5pt]
    p+q \ar[r,mapsto] & \textnormal{Span}(p,q),
    \end{tikzcd}
\end{center}
which is a morphism when $C\subseteq \mathbb{P}$ has no 3-secant-line. If $C$ admits a $4$-secant-$2$-plane $\Lambda \subseteq \mathbb{P}$ --- an obstruction to $3$-very-ampleness of $\mathcal{O}_C(1)$ --- then by fixing one of the points $x \in C\cap \Lambda = \{x_1,x_2,x_3,x_4\}$ we obtain a pencil of lines in $\Lambda$ containing $x$. Three of these are the spans of $\{x_i,x\}$. This pencil thus corresponds to a $3$-secant-line of $C_2 \subseteq \mathbb{P}\wedge^2V$ --- an obstruction to $2$-very-ampleness of $\mathcal{O}_{C_2}(1)$. This construction can be reversed by using the fact that the Grassmannian is cut out by quadrics.
\end{examplenn}

While there has been much study of projective normality and property $N_p$ more generally on surfaces of general type, often the line bundles considered are more positive than the canonical bundle. A selection of the vast body of work addressing this theme includes \cite{gall-purna,han-purna,muk-ray,purna}. Work on the canonical map itself is also extensive and includes \cite{beauville,bom1,bom2,catanese,ciliberto,franciosi,konno,mendes-lopes,mumford}, with perhaps among the strongest results related to the present work being that of Ciliberto in \cite{ciliberto} which indicates that the canonical ring of a surface of general type is generated in degree $\leq 5$ (where projective normality corresponds to generation in degree $1$). This was recovered via different techniques in \cite{gall-purna}. In a different related direction, the interesting recent work \cite{ein.niu.park} studies the syzygetic structure of secant varieties of high degree curves.

Our strategy in establishing Theorem \ref{more-general} combines a few different techniques. We begin with a standard regularity argument, reducing the issue to establishing $d$-normality (see Lemma \ref{regularity}) of the relevant line bundle on $C_2$ for $d = 2,3,4$. For $d = 2$ the argument then uses the quadric generation assumption on the homogeneous ideal $I_C$ directly, via some representation-theoretic comparisons between the line bundles on $C$ and $C_2$ respectively. For $d = 3$ it suffices to establish certain cohomology vanishings for kernel bundles on $C_2$, which we obtain using a couple of standard vector bundle techniques --- namely finding an appropriate filtration and demonstrating stability of a relevant symmetric power. For $d = 4$ it is enough to use Green's duality and vanishing theorems for Koszul cohomology coming from \cite{green}.

The structure of this note is as follows. In \S\ref{setup} we establish notation, recall some standard ideas connected to symmetric products and note a cohomology vanishing statement to be used later. In \S\ref{proof-main} we recall the regularity arguments used in conjunction with the associated theorem of Mumford to reduce Theorem \ref{more-general} to a handful of surjections, which we then prove up to granting Lemmas \ref{3-4-norm} and \ref{stability}. Finally, in \S \ref{filtering} and \S \ref{stability-proof} respectively, we prove these lemmas. We work throughout over $\mathbb{C}$. We are grateful to Rob Lazarsfeld for valuable comments.

\numberwithin{equation}{section}

\section{Setup}\label{setup}

Let $C$ denote a smooth projective curve of genus $g$ over $\mathbb{C}$ and let $K$ denote its canonical line bundle.
\begin{definition}
The \textit{Clifford index} of $C$ is defined to be \[\textnormal{Cliff}(C) := \textnormal{min}\{\hspace{1pt}\text{deg}(L) - 2\hspace{1pt}r(L) \hspace{1pt} : \text{$L$ is a globally generated line bundle of degree $\leq g-1$}\hspace{1pt}\}.\] 
\end{definition}

For $L$ a line bundle on $C$, the following property measures the syzygies of its section ring $R := \oplus_n H^0(C,L^n)$ as a graded module over the polynomial algebra $S:=\textnormal{Sym}\hspace{1pt}H^0(C,L)$.
\begin{definition}[{\cite[\S 3a]{green-laz}}]\label{np}
We say that $L$ \emph{has property $N_p$} if, in the minimal graded-free resolution of its section ring \[\cdots \lra E_1 \lra E_0 \lra R,\] the $S$--module $E_i$ is concentrated in degrees $\leq i + 1$ for all $i \leq p$.
\end{definition}
Thus property $N_0$ means $L$ determines a projectively normal embedding in a projective space $\mathbb{P}$, property $N_1$ means $L$ has $N_0$ and the homogeneous ideal $I_{C/\mathbb{P}}$ is generated by quadrics, property $N_2$ means $L$ has $N_1$ and the relations between the quadric generators are linear. In general we say property $N_p$ means the section ring $R$ \textit{has linear syzygies to order $p$.}

\subsection{Symmetric products}\label{setup-sym}
With $C$ as above, let $S := C_2$ denote the symmetric square of $C$. The following diagram is the key to most of the comparisons we will make between $C$ and its symmetric product:
\begin{center}
    \begin{tikzcd}
    & C\times C\ar[dl,"p_i"']\ar[dr,"\pi"]\\
    C & & S.
    \end{tikzcd}
\end{center}
Here $p_i$ denotes projection to the $i^{\text{th}}$ factor ($i=1,2$) and $\pi$ denotes the quotient map $(p,q) \mapsto p+q$ which ramifies simply along the diagonal \[\Delta := \{(p,p) \in C\times C\}.\]

With $L$ as above, we define the associated tautological bundle on $S$
\begin{align*}
    E_L & := \pi_*p_i^*L
\end{align*}
noting in particular that the cotangent bundle of $S$ is tautological
\begin{equation}\label{cotan-taut}
    \Omega_S^1 \cong E_K.
\end{equation}
We also note that the natural involution on $C\times C$ induces an involution on the rank 2 bundle $\pi_*(L\boxtimes L)$ which subsequently splits
\begin{align*}
\pi_*(L\boxtimes L) \cong T_L\oplus N_L
\end{align*}
into invariant and anti-invariant line bundles. The symmetry is naturally reflected in the global sections:
\begin{equation}
\begin{split}
    H^0(S,E_L) & \cong H^0(C,L)\\
    H^0(S,T_L) & \cong S^2H^0(C,L)\\
    H^0(S,N_L) & \cong \wedge^2H^0(C,L).
\end{split}
\end{equation}
Since the $\mathbb{Z}/2$-action on $C\times C$ induces an isomorphism $\pi_*(\mathcal{O}\boxtimes L) \cong \pi_*(L\boxtimes \mathcal{O})$, we can take the direct sum of evaluation maps along $\pi$ to obtain an injection \[\pi^*E_L \hookrightarrow L\boxplus L\]
which drops rank along $\Delta$ thereby yielding
\begin{align*}
    \text{det}(\pi^*E_L) \cong \text{det}(L\boxplus L)(-\Delta).
\end{align*}
From this, the projection formula, and the fact that $T_L$ is the invariant sub--bundle, we readily identify the pullbacks
\begin{align*}
    \pi^*T_L & \cong L\boxtimes L\\
    \pi^*N_L & \cong (L\boxtimes L)(-\Delta)
\end{align*}
as well as
\begin{align*}
    N_L \cong \text{det}(E_L).
\end{align*}

Finally, we note that any point $x \in C$ determines an embedding \[\iota_x : C \hookrightarrow S\] given by $p \mapsto x + p$ and we denote the image by $C_x$. One sees that
$$
\mathcal{O}_S(-C_x) \cong T_{\mathcal{O}_C(-x)}.
$$

\subsection{A cohomology vanishing}
We conclude this section by establishing a cohomology vanishing result that we will use later. To prepare for it, we note

\begin{proposition}
The bundle $E_L$ is globally generated if and only if $L$ is very ample.
\end{proposition}
\begin{proof}
At a point $p+q \in S$ the evaluation map \[H^0(S,E_L)\otimes \mathcal{O}_S \lra E_L\] coincides with \[H^0(C,L)\lra H^0(C,L\otimes \mathcal{O}_{p+q}).\]
Surjectivity of the latter means that global sections of $L$ separate the points $p$ and $q$ (or their tangents, if $p = q$).
\end{proof}

\begin{remark}\label{v-ample-sub}
Indeed, by the same reasoning, $E_L$ is generated by a subspace $V\subseteq H^0(S,E_L)\cong H^0(C,L)$ if and only if $|V| \subseteq |L|$ is very ample on $C$.
\end{remark}

\begin{proposition}\label{vanishing}
Let $L$ and $A$ be line bundles on $C$ for which \[H^1(C,K\otimes A\otimes L^{-1}) = 0.\] If the multiplication map \[H^0(C,K\otimes A)\otimes H^0(K\otimes A\otimes L^{-1}) \lra H^0(C,K^2\otimes A^2\otimes L^{-1})\] is surjective, then we have the vanishing \[H^1(S,E_L\otimes N_A^{-1}) = 0.\]
\end{proposition}
\begin{proof}
Note that by definition of $E_L$ and the projection formula along $\pi$, we have
\begin{align*}
    H^1(S,E_L\otimes N_A^{-1})
    & \cong H^1(C\times C,(A^{-1}\boxtimes (L\otimes A^{-1}))(\Delta))
\end{align*}
since $\pi$ is finite. By Serre duality (and the fact that $K_{C\times C} \cong K\boxtimes K$),
\begin{align*}
    H^1(C\times C,(A^{-1}\boxtimes (L\otimes A^{-1}))(\Delta)) & \cong H^1(C\times C,((K\otimes A)\boxtimes (K\otimes A\otimes L^{-1}))(-\Delta))^{\vee}
\end{align*}
and the latter group can be seen to vanish given the stated conditions by twisting the short exact sequence \[0 \lra \mathcal{O}_{C\times C}(-\Delta) \lra \mathcal{O}_{C\times C} \lra \mathcal{O}_{\Delta} \lra 0,\]
by $(K\otimes A)\boxtimes (K\otimes A\otimes L^{-1})$ and taking the associated long exact sequence of cohomology.
\end{proof}

\section{Proof of Theorems \ref{main-theorem} and \ref{more-general}}\label{proof-main}

In this section, we prove the normal generation statements of our two main theorems up to a small collection of cohomology vanishings to be established afterwards.

\begin{lemma}\label{regularity} For any smooth projective surface $S$ with very ample line bundle $B$, suppose $H^1(S,B^{m-2}) = H^2(S, B^{m-3}) = 0$ for some integer $m$. Then normal generation of $B$ follows from $d$-normality, i.e. surjection of the multiplication map \[ S^dH^0(S,B) \lra H^0(S,B^d),\] for $d = 2,\ldots,m-1$.
\end{lemma}
\begin{proof}
The complete embedding of $S$ in projective space $\mathbb{P}$ by $B$ determines an ideal sheaf $\mathcal{I}_{S/\mathbb{P}}$ for which we have
\begin{align*}
    H^2(\mathbb{P},\mathcal{I}_{S/\mathbb{P}}(m-2)) & \cong H^1(S,B^{m-2})\\
    H^3(\mathbb{P},\mathcal{I}_{S/\mathbb{P}}(m-3)) & \cong H^2(S,B^{m-3}).
\end{align*}
Mumford's theorem for regularity of sheaves on projective space (see e.g. \cite[Theorem 1.8.3]{laz}) implies that \[H^1(\mathbb{P},\mathcal{I}_{S/\mathbb{P}}(d)) = 0\] for all $d \geq m$ whenever
\[H^i(\mathbb{P},\mathcal{I}_{S/\mathbb{P}}(m-i)) = 0\] for all $i > 0$.
\end{proof}

\begin{lemma}\label{2-norm}
When a very ample line bundle $L$ on a smooth projective curve $C$ determines an embedding $C \subset \mathbb{P}$ whose homogeneous ideal is generated in degree 2, the line bundle $B = N_L$ on $S = C_2$ is 2-normal.
\end{lemma}
\begin{proof}
With our hypothesis of generation of the homogeneous ideal in degree 2, the proof of \cite[Proposition 9.13]{thesis} applies to the present situation even without the $\text{deg}(L) \geq 2g+2$ condition assumed there since the degree 2 generation of the ideal is the essential part.
\end{proof}

The ``if'' direction of Theorem \ref{main-theorem} follows from Theorem \ref{more-general} by taking $L = K$, so we present the proof of the latter first. Afterwards we give the shorter argument for the ``only if" direction of Theorem \ref{main-theorem}.

\begin{proof}[Proof of Theorem \ref {more-general}]
As before let $K$ denote the canonical on $C$ and let $S = C_2$. Surjectivity of the multiplication map \[\mu_n : H^0(S,N_L)\otimes H^0(S,N_L^n)\lra H^0(S,N_L^{n+1})\] together with $n$-normality of $N_L$ will imply $(n+1)$-normality by the commutative diagram
\begin{center}
    \begin{tikzcd}
    H^0(S,N_L)\otimes S^nH^0(S,N_L) \ar[r] \ar[d] & S^{n+1}H^0(S,N_L)\ar[d]\\
    H^0(S,N_L)\otimes H^0(S,N_L^n) \ar[r] & H^0(S,N_L^{n+1}).
    \end{tikzcd}
\end{center}
Note also that $H^1(S,N_L^3) = H^2(S,N_L^2) = 0$ since by \S\ref{setup} we can write, for any $k$, \[N_L^k \cong K_S\otimes N_{L^2\otimes K^{-1}} \otimes N_L^{k-2}\] to which Kodaira vanishing can be applied for $k\geq 2$ as long as $N_{L^2\otimes K^{-1}}$ is still ample --- and this follows for the $L$ under consideration since \cite[\S 5]{kouvidakis} says $N_B$ is ample as soon as $\text{deg}(B) - 1 - g > 0$.

So by Lemma \ref{regularity} we are done if we establish $d$-normality of $N_L$ for $d = 2,3,4$ --- that is, surjectivity of $\mu_n$ for $n = 1,2,3$.

For $d = n+1 = 4$ this follows from from an application of Koszul duality: by \cite[Theorem (2.c.6)]{green} the cokernel coincides with a Koszul cohomology group \[\text{coker}(\mu_4)^{\vee} \cong K_{r-2,-1}(S,K_S;N_L)\]
where $r = h^0(S,N_L)-1$, and by \cite[Theorem (3.a.1)]{green} the latter vanishes since
\begin{align*}
    h^0(S,E\otimes N_L^q) \leq p
\end{align*}
for $E = K_S = N_K$, $q = -1$ and $p = r-2$. To see this, note that by \S\ref{setup}
\begin{align*}
    N_K\otimes N_L^{-1} & \cong T_{K\otimes L^{-1}}\\
    \\
    \implies \text{dim }H^0(S,N_K\otimes N_L^{-1}) & = \text{dim }S^2H^0(C,K\otimes L^{-1})\\
    \\
    & = {h^1(C,L) + 1\choose 2} \leq 1
\end{align*}
since the assumptions in Theorem \ref{more-general} force $h^1(C,L) \leq 1$.

For $d = n+1 =3$ we use the fact that the evaluation map of $N_L$ sits in the following exact sequence \[0 \lra S^2M_{E_L} \lra H^0(S,E_L)\otimes M_{E_L} \lra H^0(S,N_L)\otimes \OO_S \lra N_L \lra 0\] where $M_{E_L}$ denotes the kernel of the evaluation map $H^0(S,E_L)\otimes \mathcal{O}_S \lra E_L$ (surjective by the fact that $L$ is very ample).
This follows, for example, from \cite[Theorem B.2.2]{laz} applied to the evaluation map of $N_L$, with $k=1$. From this it is clear that for any $m$, surjectivity of the multiplication map $\mu_m$ will follow from the vanishings
\begin{align*}
    H^1(S,M_{E_L}\otimes N_L^m) & = 0\\
    H^2(S,S^2M_{E_L}\otimes N_L^m) & = 0
\end{align*}
by \cite[Proposition B.1.1]{laz}. 
We prove these vanishings for $m = n = 2$ in Lemmas \ref{3-4-norm} and \ref{stability} (following).

For $d = n + 1 = 2$ the cohomology vanishing approach does not work (indeed for $L = K$ we find $H^1(S,M_{E_L}\otimes N_L)\not= 0$) and the Koszul duality theorem does not help. Instead we invoke Lemma \ref{2-norm}, which is equivalent to the present case ($d = n+1 = 2$), and to demonstrate the more direct approach used to establish it, we recall the key idea in that proof here: consider the following commutative diagram with exact rows
\begin{equation}\label{diag-key}
    \begin{tikzcd}
    & S^2\wedge^2H^0(C,L) \ar[r] \ar[d,"a"] & S^2S^2H^0(C,L) \ar[r]\ar[d,"b"] & S^4H^0(C,L)\ar[d,"c"] \ar[r] & 0\\
    0 \ar[r] & \text{ker}(m) \ar[r] & S^2H^0(C,L^2) \ar[r,"m"] & H^0(C,L^4)
    \end{tikzcd}.
\end{equation}
Since $b$ is surjective when $L$ is 2-normal, the Snake Lemma implies that $\text{coker}(a) = 0$ if and only if the map \[\begin{tikzcd}\text{ker}(b)\ar[r] & \text{ker}(c)\end{tikzcd}\] is a surjection. But we note
\begin{itemize}
    \item[$(i)$] with $I_C$ denoting the homogeneous ideal of the projective embedding determined by $|L|$, the maps \[\begin{tikzcd}I_{C,2}\otimes S^2H^0(C,L) \ar[d]\ar[dr,"\varphi"] \\\text{ker}(b)\ar[r] & \text{ker}(c) = I_{C,4} \end{tikzcd}\]
    commute;
    \item[$(ii)$] from \S\ref{setup-sym} we have  \[S^2H^0(C,L^2) \cong H^0(S,T_{L^2})\] and, in addition, with $\delta \subseteq S$ denoting the image $\pi(\Delta) \cong C$ of the diagonal $\Delta \subseteq C\times C$,
    \[T_{L^2}\vert_{\delta} \cong L^4\]
    so that the map $m$ in Diagram \ref{diag-key} coincides with restriction of global sections of $T_{L^2}$ to $\delta$ on $S$ and therefore \[\text{ker}(m) \cong H^0(S,T_{L^2}(-\delta));\]
    \item[$(iii)$] it also follows from \S\ref{setup-sym} that \[N_L^2 \cong T_{L^2}(-\delta)\] and, with some work (again, see \cite[\S 9]{thesis} for details), one finds that the map $a$ in \ref{diag-key} in fact coincides with the multiplication map \[S^2H^0(S,N_L) \lra H^0(S,N_L^2)\] through which $\mu_1$ clearly factors.
\end{itemize}
Putting these three observations together we see that, under the 2-normality assumption on $L$, we obtain surjectivity of $\mu_1$ when $I_C$ is generated by its degree 2 component.

\end{proof}

\begin{proof}[Proof of Theorem A]
As mentioned above, it remains to establish the ``only if" direction of the statement --- that if $K_S$ determines a projectively normal embedding of $S$, then $\text{Cliff}(C) \geq 2$. This is immediate from classical results, yet for completeness we present the argument: already the very ampleness of $K_S = N_K$ (see Equation \ref{cotan-taut}) implies $C$ can not be hyperelliptic --- every fiber $\{p,q\}\subset C$ of the hyperelliptic map would determine a basepoint $p+q \in S$ of $K_S$. Thus $K$ is very ample on $C$ and so then by \cite[Main Theorem]{catanese-goettsche} the very ampleness of $N_K$ is equivalent to $2$-very-ampleness of $K$ on $C$. This rules out the possibility that $C$ is trigonal, so either $C$ is a smooth plane quintic or $\text{Cliff}(C) \geq 2$. But by the $d = n + 1 = 2$ part of the proof of Theorem \ref{more-general} above, 2--normality of $K_S$ implies that the multiplication \[I_{C,2}\otimes S^2H^0(C,K) \lra I_{C,4}\] is surjective --- this does not happen if $C$ is a smooth plane quintic.
\end{proof}

It remains to establish each of the following:

\begin{lemma}\label{3-4-norm}
When $C$ contains a reduced effective divisor $D$ of degree $\textnormal{deg}(D) \leq \textnormal{deg}(L) - 1 - g$ which imposes independent conditions on $L$, for which $L(D)$ is non-special and basepoint free, and for which $L(-D)$ remains very ample, we have the cohomology vanishing \[H^1(S,M_{E_L}\otimes N_L^2) = 0.\]
\end{lemma}

\begin{lemma}\label{stability}
For $L$ as in Theorem \ref{more-general}, if the kernel bundles $M_{L(-p)}$ are \emph{semi-stable} for a general $p \in C$, then we have the cohomology vanishing \[H^2(S,S^2M_{E_L}\otimes N_L^2)=0.\]
\end{lemma}

Each of these Lemmas applies in the setting of Theorem \ref{more-general} --- in the first, taking $D = p+q$ for almost (see Remark \ref{rem-glob} below) any pair of distinct points $p,q \in C$ suffices since the assumptions on $L$ imply it is $2$-very-ample; in the second, the semi-stability hypothesis follows from \cite[Theorem 1.3]{camere}.

\begin{remark}\label{rem-glob}
Effectiveness of $D$ and global generation of $L$ imply, respectively, the inequalities and the equality in
\begin{equation}\label{rmk-eqn}
h^1(C,L(D)) \leq h^1(C,L) = h^1(C,L(-x)) \geq h^1(C,L(D)(-x)).
\end{equation}
Therefore when $L$ is non-special, all of the above are zero and so $L(D)$ is also globally generated. However, if $L$ is special one needs to rule out the possibility that $h^1(C,L(D)) = 0$ while the rest are $1$ (note that the assumptions of Theorem \ref{more-general} imply $h^1(C,L) \leq 1$). This can be done by choosing $D$ to be disjoint from the unique element of $|K\otimes L^{-1}|$ since then both \[h^1(C,L(D)) = h^0(C,K\otimes L^{-1}(-D)) = 0\] and \[h^1(C,L(D)(-x)) = h^0(C,K\otimes L^{-1}(x)(-D)) = 0\] for any $x \in C$ (including $x \in D$ since $\text{deg}(D) = 2$). Thus $L(D)$ is also globally generated.

\end{remark}

\section{Filtering the tautological kernel bundle}\label{filtering}

In this section we prove Lemma \ref{3-4-norm} by producing a filtration of $M_{E_L}$ inspired by standard arguments used in the case of curves.\\
\\
With notation as in previous sections, suppose $D \subseteq C$ is a reduced effective divisor for which $A := L(-D)$ remains very ample. Then we have an exact sequence \[0 \lra A \lra L \lra L\vert_D \lra 0\] to which we can apply the operation $\pi_*p_i^*(\underline{\hspace{6pt}})$, obtaining \[0 \lra E_A \lra E_L \lra \pi_*(p_i^*(L\vert_D)) \cong \oplus_{x \in D}\mathcal{O}_{C_x} \lra 0\] whose right--most map is surjective since $\pi$ is finite. Providing that global sections of $L$ separate points of $D$ we can write \[H^0(L)/H^0(A) \cong \oplus_{x \in D} H^0(L\vert_x)\] and obtain the following diagram
\begin{equation}\label{diag}
    \begin{tikzcd}
    & 0\ar[d] & 0\ar[d] & 0\ar[d]\\
    0 \ar[r] & M_{E_A} \ar[r] \ar[d] & H^0(C,A)\otimes \mathcal{O}_S \ar[r,"a"] \ar[d] & E_A\ar[d]\\
    0 \ar[r] & M_{E_L} \ar[r] \ar[d,"b"] & H^0(C,L)\otimes \mathcal{O}_S \ar[r] \ar[d] & E_L \ar[r] \ar[d] & 0\\
    0 \ar[r] & \mathcal{K} \ar[r] & \oplus H^0(L\vert_x)\otimes \mathcal{O}_S \ar[r] \ar[d] & \oplus \mathcal{O}_{C_x}\ar[d]\ar[r] & 0\\
    &  & 0 & 0
    \end{tikzcd}.
\end{equation}
By the Snake Lemma, the map $b$ will be surjective if $A = L(-D)$ remains very ample (since this is equivalent to surjectivity of $a$). By Remark \ref{v-ample-sub} a similar diagram can be constructed when $|V|$ is a very ample sub-linear series of $|L|$ and so, successively choosing $A = L(-D)$ and $V \subseteq H^0(C,L(-D))$ very ample with $\text{dim}(V) = 4$, we obtain exact sequences
\begin{center}
    \begin{tikzcd}[row sep = tiny]
    0 \ar[r] & M_{E_A} \ar[r] & M_{E_L} \ar[r] & \mathcal{K}\cong\oplus\mathcal{O}_S(-C_x) \ar[r] & 0\\
    0 \ar[r] & M_{V,E_A} \cong (E_A)^* \ar[r] & M_{E_A} \ar[r] & (H^0(L(-D))/V) \otimes \mathcal{O}_S \ar[r] & 0
    \end{tikzcd}.
\end{center}
From these it is clear that establishing vanishings
\begin{align*}\label{eqn}
    H^1(S,N_L^2) & = 0\\
    H^1(S,(E_A)^*\otimes N_L^2) & = 0 \tag{*}\\
    \text{and}\quad H^1(S,N_L^2\otimes \mathcal{O}_S(-C_x)) & = 0 
\end{align*}
(for $x \in D$) would imply that \[H^1(S,M_{E_L}\otimes N_L^2) = 0.\] This allows:

\begin{proof}[Proof of Lemma \ref{3-4-norm}]
With the setup outlined above, the result will follow from establishing the vanishings (*).
Using the observations in \S\ref{setup-sym}, we have
\begin{align*}
    N_L^2 & \cong N_K\otimes T_{L\otimes K^{-1}}\otimes N_L\\
    & \cong K_S \otimes N_{L^2\otimes K^{-1}}\\
    N_L^2\otimes \mathcal{O}_S(-C_x) & \cong N_L^2\otimes T_{\mathcal{O}_C(-x)}\\
    & \cong K_S\otimes N_{L^2\otimes K^{-1}(-x)}
\end{align*}
and
\begin{align*}
    H^1(S,(E_A)^*\otimes N_L^2)^{\vee} & \cong H^1(S,E_A\otimes N_L^{-2}\otimes K_S)\\
    & \cong H^1(S,E_A\otimes N_{L^2\otimes K^{-1}}^{-1}).
\end{align*}
By \cite[\S 5]{kouvidakis} $N_B$ is ample when $\text{deg}(B) - 1 - g > 0$ so, since $\text{deg}(L^2\otimes K^{-1}(-x)) \geq g+2$, the above shows that
\begin{align*}
    H^1(S,N_L^2) = H^1(S,N_L^2\otimes \mathcal{O}_S(-C_x)) = 0
\end{align*}
by Kodaira vanishing. For the remaining vanishing, by Proposition \ref{vanishing} and the fact that $L(D)$ is non-special, we just require that the map \[H^0(C,L^2)\otimes H^0(C,L(D)) \lra H^0(C,L^3(D))\] is surjective. But, using the effectiveness of $D$ and the assumption on its degree, we have \[h^1(C,L(-D)) \leq \textnormal{deg}(D) + h^1(C,L) \leq h^0(C,L) - 2 \leq h^0(C,L(D)) - 2,\] and so surjectivity follows from \cite[Theorem (4.e.1)]{green} since $L(D)$ is assumed basepoint free. So after all we have \[H^1(S,E_A\otimes N_{L^2\otimes K^{-1}}^{-1}) = 0,\] completing the proof.

\end{proof}

\section{Stability of the kernel bundle}\label{stability-proof}

In this section we prove Lemma \ref{stability} by establishing slope semi-stability of the bundle $S^2M_{E_L}$ with respect to the natural ample class on $S$. Again keeping notation as in previous sections, we begin with a brief summary of the numerics on $S$ that we will need for the stability statement.\\
\\
Define the N\'{e}ron--Severi class \[x := c_1(\mathcal{O}_S(C_q)) \in \textnormal{NS}(S)\] for $q \in C$ (any point $q$ will determine the same class). We will abuse notation slightly and denote by \[\delta \in \textnormal{NS}(S)\] the class of $\mathcal{O}_S(\delta)$ for $\delta = \pi(\Delta)$ defined above. We have (e.g. from \cite[Lemmas 1 \& 7]{kouvidakis}) the intersection products \[ x^2 = x\cdot (\delta/2) = 1;\quad (\delta/2)^2 = 1-g\] where $g = g(C)$. By \S\ref{setup-sym}, letting $d := \text{deg}(L)$, we have
\begin{align*}
    c_1(T_L) & = dx\\
    c_1(N_L) & = dx - \delta/2.
\end{align*}
Moreover, it is well-known that $x$ is an ample class on $S$ and so we define $\mu$ to be the slope function, relative to $x$, for coherent sheaves $\mathcal{F}$ on $S$ i.e. \[\mu(\mathcal{F}) := \frac{c_1(\mathcal{F})\cdot x}{\text{rk}(\mathcal{F})}.\]

With these things in place, we are almost ready to prove Lemma \ref{stability} --- the last thing before doing so is to establish
\begin{lemma}\label{taut-ker-stab}
The bundle $S^2M_{E_L}$ on $S$ is $\mu$-slope semi-stable provided that $M_{L(-p)}$ is semi-stable on $C$ for a general point $p \in C$.
\end{lemma}
\begin{proof}
We adapt the proof of \cite[Theorem 3.8]{mistretta}: specifically, note that if a sub-sheaf $F\subseteq S^2M_{E_L}$ (reflexive, without loss of generality) were $\mu$-destabilizing, then
\begin{equation}\label{stab-eq}
\mu(F) > \mu(S^2M_{E_L}) = 2\mu(M_{E_L}).
\end{equation}
For a general choice of $p \in C$, restricting the inclusion $F\subseteq S^2M_{E_L}$ to $C_p \subseteq S$ yields a sub--bundle
\[F\vert_{C_p} \subseteq S^2M_{E_L}\vert_{C_p}.\] However, by \cite[Lemma 3.7]{mistretta} we have \[M_{E_L}\vert_{C_p}\cong M_{L(-p)},\] and since $x = c_1(\mathcal{O}_S(C_p))$ is the class defining the slope $\mu$ we have
$$
\mu(F) = \mu(F\vert_{C_p});\quad \mu(M_{E_L}) = \mu(M_{L(-p)})
$$
(here we abuse notation slightly by using $\mu$ to also denote slope on the curve $C_p \cong C$). Thus Equation \ref{stab-eq} would imply $S^2M_{L(-p)}$ is unstable on $C$. Since we are in characteristic 0, however, the symmetric power $S^2M_{L(-p)}$ is semi-stable whenever $M_{L(-p)}$ is (this implication is well-known, but see e.g. \cite[Corollary 6.4.14]{pos2} or \cite[Theorem I.10.5]{hartshorne} for details).
\end{proof}

\begin{proof}[Proof of Lemma \ref{stability}]
For the sake of contradiction, suppose we have \[H^2(S,S^2M_{E_L}\otimes N_L^2)\not= 0.\]
Then by Serre duality we have an injection \[\begin{tikzcd}N_{L^2\otimes K^{-1}}\cong N_L^2\otimes K_S^{-1} \ar[r,hookrightarrow] & (S^2M_{E_L})^{\vee} \cong S^2(M_{E_L})^{\vee}\end{tikzcd}\] and so, by the $\mu$-slope semi-stability supplied by Lemma \ref{taut-ker-stab}, this implies that
\begin{align*}
\mu(N_{L^2\otimes K^{-1}}) & \leq \mu(S^2(M_{E_L})^{\vee}) = 2\hspace{1pt}\mu((M_{E_L})^{\vee}).
\end{align*}
Using the numerics summarized above this means
\begin{align*}
    2d-(2g-2) - 1 & \leq \frac{2(d - 1)}{h^0(C,L)-2},
\end{align*}
or equivalently,
\begin{align*}
    h^0(C,L) & \leq 3 + \frac{2g-5}{2(d-g)+3}
\end{align*}
which is incompatible with $h^0(C,L) \geq d+1-g$ for the range of $d$ considered.
\end{proof}

\end{document}